\documentclass{amsart}
\usepackage{amsfonts,amssymb,amscd,amsmath,enumerate,verbatim,calc}
\usepackage{amscd,amssymb,amsopn,amsmath,amsthm,graphics,amsfonts,enumerate,verbatim,calc}
\usepackage[dvips]{graphicx}
\usepackage[colorlinks=true,linkcolor=red,citecolor=blue]{hyperref}
\input xy
\xyoption{all}
\calclayout

\newcommand{\rt}{\rightarrow}
\newcommand{\lrt}{\longrightarrow}

\newcommand{\st}{\stackrel}

\newcommand{\la}{\lambda}
\newcommand{\La}{\Lambda}

\newcommand{\lan}{\langle}
\newcommand{\ran}{\rangle}

\newcommand{\C}{\mathbb{C} }
\newcommand{\D}{\mathbb{D} }

\newcommand{\N}{\mathbb{N} }
\newcommand{\Z}{\mathbb{Z} }

\newcommand{\CA}{\mathcal{A} }

\newcommand{\CI}{\mathcal{I} }

\newcommand{\CM}{\mathcal{M} }

\newcommand{\CT}{\mathcal{T} }
\newcommand{\CX}{\mathcal{X} }

\newcommand{\mmod}{{\textsf{mod}}}

\newcommand{\Prj}{{\textsf{Proj}}}

\newcommand{\GPrj}{{\textsf{GProj}}}
\newcommand{\Gprj}{{\textsf{Gproj}}}

\newcommand{\im}{{\rm{Im}}}

\newcommand{\add}{{\rm{add}}}

\newcommand{\id}{{\rm{id}}}

\newcommand{\Ggldim}{{\rm{Ggldim}}}

\newcommand{\derdim}{{\rm{der.dim}}}

\newcommand{\gldim}{{\rm{gl.dim}}}

\newcommand{\Gpd}{{\rm{Gpd}}}

\newcommand{\Ker}{{\rm{Ker}}}

\newcommand{\Hom}{{\rm{Hom}}}

\newcommand{\End}{{\rm{End}}}

\newtheorem{theorem}{Theorem}[section]
\newtheorem{corollary}[theorem]{Corollary}
\newtheorem{lemma}[theorem]{Lemma}
\newtheorem{proposition}[theorem]{Proposition}

\theoremstyle{definition}
\newtheorem{definition}[theorem]{Definition}

\newtheorem{remark}[theorem]{Remark}
\newtheorem{s}[theorem]{}

\theoremstyle{plain}

\theoremstyle{definition}

\numberwithin{equation}{section}

\begin{document}

\title[Derived dimension of abelian categories]{On the derived dimension of abelian categories}

\author[Asadollahi and Hafezi]{Javad Asadollahi and Rasool Hafezi}

\address{Department of Mathematics, University of Isfahan, P.O.Box: 81746-73441, Isfahan, Iran and School of Mathematics, Institute for Research in Fundamental Science (IPM), P.O.Box: 19395-5746, Tehran, Iran } \email{asadollahi@ipm.ir}
\email{r.hafezi@sci.ui.ac.ir}



\thanks{The research of authors was in part supported by a grant from IPM (No. 90130216).}

\begin{abstract}
We give an upper bound on the dimension of the bounded derived category of an abelian category. We show that if $\CX$ is a sufficiently nice subcategory of an abelian category, then derived dimension of $\CA$ is at most $\CX$-dim$\CA$, provided $\CX$-dim$\CA$ is greater than one. We provide some applications.
\end{abstract}

\maketitle

\section{Introduction}
Motivated by the work of Bondal and Van den Berg, Rouquier attached a numerical invariant to a triangulated category; its dimension. This dimension gives a new invariant for algebras and algebraic varieties under derived equivalences. Roughly speaking, it measures the minimum steps one requires to build the whole triangulated category $\CT$ out of a single object $M$ for details see \cite{Rq2}. Rouquier's results enabled him to give the first example of artin algebras of representation dimension greater than three \cite{Rq1}. Therefore he could solve a long standing open problem started with the Auslander's work in his Queen Mary notes in 1971.

The notion of dimension of triangulated categories then took into account by several authors and has found several applications in different subjects ranging from algebraic geometry to representation theory of algebras and to commutative ring theory. In the setting of artin algebras, it provides a lower bound for the representation dimension of an artin algebra. Moreover, it is related to some known invariants like Lowey length and global dimension, see e.g. \cite{ABIM}, \cite{KK}, \cite{Opp}, \cite{O} and \cite{Rq1}.

Let $\CA$ be an abelian category. It is shown by Han \cite{H} that if $\CA$ is of finite representation type, that is, if there exists an object $M$ in $\CA$ with $\CA=\add M$, then $\dim\D^b(\CA)\leq 1$, where $\D^b(\CA)$ denotes the bounded derived category of $\CA$. The invariant $\dim\D^b(\CA)$ is known as the derived dimension of $\CA$, denoted by $\derdim\CA$. Krause and Kussin \cite[Proposition 3.4]{KK} showed that if $\CA$ is an abelian category with a representation generator $M$, then $\dim\D^b(\CA)\leq \gldim\End_{\CA}(M)$. An $\add M$-resolution of $X$ is an exact sequence $\cdots \rt M_n \rt \cdots \rt M_1 \rt M_0 \rt X \rt 0$ such that $M_i\in \add M$ and the sequence is exact with respect to the functor $\Hom(M, \ )$. An object $M$ in $\CA$ is called a representation generator if every object $X \in \CA$ admits an $\add M$-resolution of finite length.

In this note, we show that if $\CX$ is a sufficiently nice subcategory of an abelian category $\CA$, which is of finite representation type, then $\dim\D^b(\CA)\leq \CX$-$\dim\CA$, provided $\CX$-dim$\CA\geq 1$. Our proof is based on the structure of projective objects in the category $\C(\CA)$, the category of complexes over $\CA$. As applications, we obtain a generalization of the main theorem of \cite{H}, provide a simple proof for the Corollary 6.7 of \cite{B} (with some weaker conditions) and deduce Proposition 2.6 of \cite{KK} as a corollary. Furthermore, we get an upper bound on the derived dimension of $n$-torsionless-finite algebras.

\section{Preliminaries}
\noindent {\bf Setup.} Throughout we assume that $\CA$ is an abelian category. We denote the category of complexes over $\CA$ by $\C(\CA)$. We grade the complexes homologically, so an object $A=(A_n,f_n)$ of $\C(\CA)$ will be written as
\[A=(A_n,f_n): \ \ \ \ \cdots \lrt A_{n+1} \st{f_{n+1}}{\lrt} A_n  \st{f_{n}}{\lrt} A_{n-1} \lrt \cdots.\]
For any integer $n$, we set $Z_nA:=\Ker f_n$ and $B_nA:=\im f_{n+1}$.\\

\s {\sc Projective objects of $\C(\CA)$.}\label{proj}
Let $\CA$ has enough projective objects. By definition, a complex $P$ in $\C(\CA)$ is a projective object if the functor $\Hom(P, \ )$ is exact, where for any complexes $A$ and $A'$, $\Hom(A,A')$ denotes the abelian group of chain maps from $A$ to $A'$. By \cite{Ro}, this is equivalent to say that $P$ is exact and $Z_nP$ is projective, for all $n \in \Z$. So, for any projective object $P$, the complex \[\cdots \rt 0 \rt P \rt P \rt 0 \rt \cdots,\]
is projective. It is known \cite{Ro} that any projective object in $\C(\CA)$ can be written uniquely as a coproduct of such complexes. \\

\s {\sc Evaluation functor and its extension.}\label{EvalFunc}
Let $i \in \Z$. There is an evaluation functor $e^i:\C(\CA) \rt \CA$, that restricts any complex $A$ to its $i$-th term $A_i$. It is well known that this functor admits a left adjoint, which we will denote by $e^i_{\la}$. In fact, this adjoint is defined as follows. Given any object $A$ of $\CA$, $e^i_{\la}(A)$ is defined to be the complex
\[\cdots 0 \rt A \st{\id}{\rt} A \rt 0 \cdots,\]
with the $A$ on the left hand side sits on the $i$-th position. It follows from the properties of the adjoint pair $(e^i_{\la},e^i)$ that if $P$ is a projective object in $\CA$ then the complex $e^i_{\la}(P)$ is a projective object in $\C(\CA)$. Moreover, by \ref{proj} any projective object in $\C(\CA)$ is a coproduct of $e^i_{\la}(P_i)$, for a family $\{P_i\}_{i \in I}$ of projective objects of $\CA$.\\

\s {\sc Dimension of triangulated categories.}
Let $\CT$ be a triangulated category and $\CI_1$ and $\CI_2$ be subcategories of $\CT$. Using the same notation as in \cite[\S 3.1]{Rq2}, we let $\CI_1 \ast \CI_2$ denote the full subcategory of $\CT$ consisting of objects $M$ such that there exists an exact triangle
\[\xymatrix{I_1 \ar[r] & M \ar[r] & I_2 \ar@{~>}[r] & },\]
in $\CT$ with $I_i \in \CI_i$, for $i=1,2$.

Let $\CI$ be a subcategory of $\CT$. Denote by $\lan\CI\ran$ the smallest full subcategory of $\CT$ that contains $\CI$ and is closed under taking finite coproducts, direct factors, and all shifts. Finally, set $\CI_1 \diamond \CI_2:=\lan\CI_1\ast \CI_2\ran$.

Inductively one defines ${\lan \CI\ran}_0 = 0$ and ${\lan \CI\ran}_n ={\lan \CI\ran}_{n-1}\diamond \lan \CI\ran$ for $n\geq1$.

\begin{definition}\cite[Definition 3.2]{Rq2}
Let $\CT$ be a triangulated category. The dimension of $\CT$, denoted by $\dim\CT$, is defined by
\[\dim\CT=\inf\{d \in \N \ | \ {\rm there \ exists} \ M \in \CT \ {\rm such \ that} \ \CT=\lan M\ran_{d+1} \}.\]
We set $\dim\CT=\infty$, if such $M$ does not exist.
\end{definition}

\s Let $\CM$ be a class of objects of an abelian category $\CA$. We denote by $\add\CM$ the objects that are isomorphic to a direct summand of a finite direct sum of objects in $\CM$. If $\CM=\{M\}$ contains a single object, we write $\add M$ instead of $\add\{M\}$. We say that a subcategory $\CX$ of $\CA$ is representation-finite if there exits an object $M \in \CX$ such that $\add\CX=\add M$.

\section{derived dimension of an abelian category}
Let us begin by a definition.

\begin{definition}
Let $\CX$ be a full subcategory of $\CA$. For any object $A \in \CA$, we define the $\CX$-dimension of $A$, denoted $\CX$-$\dim A$, as follows: if $A \in \CX$, then we set $\CX$-$\dim A=0$. If $t\geq 1$, then we set $\CX$-$\dim A\leq t$ if there exists an exact sequence
\[0 \rt X_t \rt \cdots \rt X_0 \rt A \rt 0\] where $X_i \in \CX$ for $0 \leq i \leq t$. We set $\CX$-$\dim A=t$ if $\CX$-$\dim A\leq t$ and $\CX$-$\dim A \nleqslant t-1$. Finally if $\CX$-$\dim A \neq t$ for any $t \geq 0$, we set $\CX$-$\dim A=\infty$. The $\CX$-global dimension of $A$, denoted $\CX$-$\dim\CA$, is defined to be $\sup\{\CX$-$\dim A  \ | \ A \in \CA \}$.
\end{definition}

Let $\CA$ be an abelian category with enough projectives. A subcategory $\CX$ of $\CA$ is said to be resolving if it contains the projective objects and is closed under extensions and kernels of epimorphisms. For our purpose in this section, some weaker conditions are enough.

\begin{definition}
Let $\CA$ be an abelian category with enough projectives. A full subcategory $\CX \subseteq \CA$ is called semi-resolving if it contains projective objects and satisfies the following property: for any short exact sequence
$ 0 \rt K \rt P \rt M \rt 0$ in $\CA$ with $P \in \Prj A $, if $M \in \CX$ then so is $K$, otherwise $\CX$-$\dim K=\CX$-$\dim M-1$.
\end{definition}

It is clear that any resolving subcategory of $\CA$ is semi-resolving.\\

The following lemma can be proved using the similar argument as in the proof of the Lemma 2.5 of \cite{KK}, see also \cite[Theorem 8.3]{C}. So we skip the proof.

\begin{lemma}\label{KrauseKussin}
Let $\CX \subseteq \CA$ be a semi-resolving subcategory of $\CA$, where $\CA$ is an abelian category with enough projectives. Let $d>0$ be an arbitrary integer. Let $A$ be a bounded complex in $\CA$ such that for all $i$, both $\CX$-$\dim B_iA$ and $\CX$-$\dim Z_iA$ are at most $d$. Then there exists an exact triangle
\[\xymatrix{L \ar[r] & P \ar[r] & A \ar@{~>}[r] & },\]
in $\D^b(\CA)$, in which $P$ is a bounded complex of projectives and $L$ is a complex such that for all $i$, both $\CX$-$\dim B_iL$ and $\CX$-$\dim Z_iL$ are at most $d-1$.
\end{lemma}

\noindent {\bf Notation.} Let $\CA$ be an abelian category. Given an object $M \in \CA$ and integer $i \in \Z$, we let $S^i(M)$ denote the complex with $M$ in the $i$-th place and 0 in the other places.\\

\begin{theorem}
Let $\CX \subseteq \CA$ be a semi-resolving subcategory of $\CA$, where $\CA$ is an abelian category with enough projectives. Assume that $\add\CX=\add M$, for some $M \in \CA$. Let $A=(A_n,f_n)$ be a bounded complex in $\CA$ such that for all $i$, both $\CX$-$\dim B_iA$ and $\CX$-$\dim Z_iA$ are at most $d$. Then $A \in \lan\CM\ran_{d+2}$, where $\CM=S^0(M) \oplus e_{\la}^0(M)$.
\end{theorem}

\begin{proof}
We use induction on $d$. To begin, assume that $d=0$. As in \cite[Theorem]{H}, consider the complexes $K=(Z_nA, 0)$ and $I=(B_{n-1}A, 0)$ to get the exact sequence
\[0 \lrt K \st{i} \lrt A \st{\bar{f}} \lrt I \rt 0\]
of complexes, where for each $n \in \Z$, $i_n:Z_nA \rt A_n$ is the natural embedding and $\bar{f}_n:A_n \rt B_{n-1}A$ is the map $f_n$ whose coimage is restricted to $B_{n-1}A$. So we get the triangle
\[\xymatrix{K \ar[r] & A \ar[r] & I \ar@{~>}[r] & },\]
in $\D^b(\CA)$. Since $X$ is a bounded complex over $\CA$, $K=\bigoplus S^n(Z_nA)$ and $I=\bigoplus S^n(B_{n-1}A)$ have only finitely many nonzero summand. Hence by definition $K$ and $I$ belong to $\lan\CM\ran$. Thus $A \in \lan\CM\ran_2$. Assume inductively that $d>0$ and the result has been proved for smaller value than $d$. Lemma \ref{KrauseKussin} implies the existence of an exact triangle
\[\xymatrix{P \ar[r] & A \ar[r] & \sum L \ar@{~>}[r] & },\]
in $\D^b(\CA)$, where $P$ is a bounded complex of projectives and $L$ is a complex such that $\CX$-$\dim B_nL \leq d-1$ and $\CX$-$\dim Z_nL \leq d-1$ for all $n$. Since $P$ is a projective complex, it is a coproduct of finitely many  complexes of the form $e_{\la}^i(Q)$, where $Q$ is a projective $R$-module. Therefore $P \in \lan\CM\ran$. Moreover, by induction assumption, $\sum L \in \lan\CM\ran_{d+1}$. Hence
\[A \in \lan\lan\CM\ran \ast \lan\CM\ran_{d+1}\ran.\]
Now the associativity of the operation $\ast$ implies that \[\lan\lan\CM\ran \ast \lan\CM\ran_{d+1}\ran= \lan\lan\CM\ran_{d+1}\ast \lan\CM\ran\ran.\] So $A \in \lan\CM\ran_{d+2}.$ This completes the inductive step and hence the proof.
\end{proof}

Here we present two applications of the above theorem. \\

$(1)$ Let $\CA$ be an abelian category. It is proved by Han \cite{H} that if there exists $M\in\CA$ such that $\CA=\add M$, then $\dim\D^b(\CA)\leq 1.$ In the following we use the above theorem to provide a generalization of this result.

\begin{proposition}\label{Han}
Let $\CX \subseteq \CA$ be a representation-finite semi-resolving subcategory of $\CA$. Then $\dim\D^b(\CA) \leq \CX{\rm-}\dim\CA+1.$
\end{proposition}

\begin{proof}
If $\CX$-$\dim\CA=\infty$, there is nothing to prove. So assume that $\CX$-$\dim\CA$ is finite, say $d$. So for any complex $A=(A_n,f_n)$ in $\CA$, we have $\CX$-$\dim Z_iX\leq d$ and $\CX$-$\dim B_iX \leq d$. Assume that $\add\CX=\add M$, for some $M \in \CA$. The above theorem now apply to show that $A \in \lan M\ran_{d+2}$. Therefore $\D^b(\CA)=\lan M\ran_{d+2}$, since $A$ was an arbitrary bounded complex.
\end{proof}

$(2)$ Let $\Omega^d(\CA)$ be the full subcategory of $\CA$ consisting of all $d$th syzygies, i.e. all objects $A$ for them there exist an exact sequence \[0 \rt A \rt P^0 \rt \cdots \rt P^{d-1},\]
where $P^i$, for $0 \leq i \leq d-1$, is projective. Note that $\Omega^d(\CA)$-$\dim\CA=d$. Although we do not know wether $\Omega^d(\CA)$ is semi-resolving or not, but the proof of the above theorem can be rewritten in order to prove the following corollary. We leave the details to the reader.

\begin{corollary}\label{syzygy}
Let $\CA$ be an abelian category with enough projectives. If $\Omega^d(\CA)$ is representation-finite, then $\dim\D^b(\CA)\leq d+1.$
\end{corollary}

\begin{remark}
We just remark that in case $\La$ is an artin algebra and $\CA=\mmod \La$, then $\Omega^d(\CA)$ is known as the subcategory of $\mmod\La$ consisting of $d$-torsionless modules. We say that $\La$ is a $d$-torsionless-finite algebra provided $\Omega^d(\CA)$ is representation-finite. The above corollary, in particular, shows that the derived dimension of $d$-torsionless-finite algebras is at most $d+1$.
\end{remark}

In the following we plan to show that, if $\CX$-$\dim\CA \geq 2$, we may get a sharper upper bound for the dimension of the derived category of $\CA$. The following lemma, which is a simple observation, plays a key role.

\begin{lemma}
Let $\CA$ be an abelian category. Let $\{M_i\}_{i=1}^n$ and $\{N_i\}_{i=1}^n$ be families of objects of $\CA$. Consider the following morphism
\[\bigoplus_{i=1}^n e_{\la}^i(M_i) \st{\varphi}{\lrt} \bigoplus_{i=1}^n e_{\la}^i(N_i)\]
of complexes and let $K=\Ker\varphi$ be its kernel. Then $K$ is quasi-isomorphic to the complex $S^0(H^0(K))$. Hence, in $\D^b(\CA)$ we have an isomorphism $K \cong S^0(H^0(K))$.
\end{lemma}

\begin{proof}
It follows by an easy diagram chasing that for any $i>0$, $H^i(K)=0$. Therefore the natural morphism $K \rt S^0(H^0(K))$ provides the desired quasi-isomorphism.
\end{proof}

\begin{theorem}\label{main}
Let $\CX \subseteq \CA$ be a representation-finite semi-resolving subcategory of $\CA$. Set $\CX$-$\dim\CA=d$. Then
\begin{itemize}
\item [$(i)$] If $d=0$ or $1$, then $\dim\D^b(\CA) \leq d+1$.
\item [$(ii)$] If $d \geq 2$, then $\dim\D^b(\CA) \leq d$.
\end{itemize}
\end{theorem}

\begin{proof}
$(i).$ If $\CX$-$\dim\CA$ is either zero or one, the result follows from Proposition \ref{Han}.

$(ii).$ Let $\CX$-$\dim\CA=d \geq 2$. Assume that $\add\CX=\add M.$ We claim that $\D^b(\CA)=\lan\CM\ran_{d+1}$, where $\CM=S^0(M)\oplus e_{\la}^0(M)$.
To prove the claim, let $A$ be a bounded complex in $\D^b(\CA)$. We may assume that
\[A \ = \ 0 \rt A_n \rt \cdots \rt A_1 \rt 0.\]
Using the same technique as in the proof of Lemma 2.5 of \cite{KK}, we get an exact sequence
\[0 \rt K^d \rt P^{d-1} \rt P^{d-2} \rt \cdots \rt P^1 \rt P^0 \rt A \rt 0,\]
of complexes such that $P^j$, for any $j \in \{0, 1, \cdots, d-1\}$ is a projective complex and $K^d$ is a bounded complex in which all terms, all kernels, all images and also all homologies belong to $\CX$. By \ref{proj}, each $P^j$ can be written as $\bigoplus_{i=1}^ne_{\la}^i(P^{j}_i)$, in which for all $i$ and $j$, $P_i^j$ is a projective object in $\CA$. Using this sequence, we obtain the following triangles in $\D^b(\CA)$
\[\xymatrix@C-0.5pc{K^d \ar[r] &  \bigoplus_{i=1}^ne_{\la}^i(P^{d-1}_i) \ar[r] & K^{d-1} \ar@{~>}[r] & },\]
\[\vdots\]
\[\xymatrix@C-0.5pc{K^1 \ar[r] &  \bigoplus_{i=1}^ne_{\la}^i(P^0_i) \ar[r] & A \ar@{~>}[r] & },\]

Therefore, since $H^0(K^d) \in \CX$, and by the above lemma, $K^d \cong S^0(H^0(K))$ in $\D^b(\CA)$, we may conclude that $K^d \in \lan\CM\ran$. Now the first triangle implies that $K^{d-1} \in \lan\CM\ran_2$. Continuing in this way, we finally get that $A \in \lan\CM\ran_{d+1}$. So the claim follows.
\end{proof}

\begin{corollary}
Let $\CA$ be an abelian category with enough projectives. Assume that $\Omega^d(\CA)$ is a representation-finite subcategory of $\CA$. Then  \begin{itemize}
\item [$(i)$] If $d=0$ or $1$, then $\dim\D^b(\CA) \leq d+1$.
\item [$(ii)$] If $d \geq 2$, then $\dim\D^b(\CA) \leq d$.
\end{itemize}
\end{corollary}

\begin{proof}
Part $(i)$ follows from the Corollary \ref{syzygy}. Part $(ii)$ can be proved following similar method as we used in the proof of the above theorem.
\end{proof}

Towards the end of the paper, we provide two applications of the above theorem, compare
\cite[Corollary 6.7]{B}.

$(1)$ Let $\CA$ be an abelian category with enough projective objects. An object $G\in \CA$ is called Gorenstein-projective, if there exits an exact sequence
\[\cdots \rt P_1 \rt P_0 \rt P_{-1} \rt \cdots\]
of projective objects such that the sequence remains exact under the functor $\CA( \ ,Q)$, for any projective object $Q$ and $G \cong \Ker(P_0 \rt P_{-1}).$ In module category, these modules was introduced by Auslander and Bridger in \cite{AB}, over commutative noetherian rings, where they called them modules of Gorenstein dimension zero. They has been generalized to arbitrary modules over arbitrary rings by Enochs and Jenda \cite{EJ}. They form the building blocks of a branch of homological algebra, known as Gorenstein homological algebra.

Let $\GPrj(\CA)$ denotes the class of all Gorenstein-projective objects. One can use this class to attach a homological dimension to any object $A$, namely the Gorenstein-projective dimension of $A$, denoted by $\Gpd A$. The Gorenstein global dimension of $\CA$ is defined to be the supremum of the Gorenstein-projective dimension of all objects and will be denoted by $\Ggldim \CA$. The above theorem implies that
\[\dim\D^b(\CA)\leq \max\{2,\Ggldim \CA\},\]
provided $\GPrj(\CA)=\add A$, for some $A \in \CA$.

$(2)$ An important special case of $(i)$, is $\CA=\mmod \La$, where $\La$ is an artin algebra. Let $\Gprj\La$ denote the full subcategory of $\mmod\La$ consisting of all Gorenstein-projective $\La$-modules.

\begin{definition}
An artin algebra $\La$ is called of finite Cohen-Macaulay type, finite CM-type for short, if the full subcategory $\Gprj\La$ is of finite representation type.
\end{definition}

If $\La$ is a self-injective algebra, then it is of finite CM-type if and only if it is of finite representation type.
Recall that an artin algebra $\La$ is called Gorenstein if $\id_{\La}\La<\infty$ and $\id\La_{\La}<\infty$. In this case it is known that $\id_{\La}\La=\id\La_{\La}$. If we let this common value to be $d$, then we call $\La$ a $d$-Gorenstein algebra. It is known that if $\La$ is a $d$-Gorenstein algebra, then $\Omega^d(\La)=\Gprj\La$. So we get the following result.

\begin{corollary}\cite[Corollary 6.7]{B}
Let $\La$ be a $d$-Gorenstein artin algebra of finite CM-type. Then $\derdim\La\leq \max\{2,d\}$.
\end{corollary}

\begin{proof}
To prove the corollary, one just should note that by our assumptions $\Omega^d(\La)=\Gprj\La.$ Now the result follows from Theorem \ref{main}.
\end{proof}

\section*{Acknowledgments}
Ryo Takahashi mentioned that T. Aihara, T. Araya, M. Yoshiwaki and himself also have a proof for Theorem \ref{main} \cite{AATY}. Their approach is completely different from us. We thank them for letting us to have the preliminary version of their preprint. The authors thank the Center of Excellence for Mathematics (University of Isfahan).

\end{document}